\documentclass{amsart}
\usepackage{amssymb,amsfonts}
\usepackage{latexsym}
\usepackage{amsmath,mathrsfs}
\usepackage[all,arc]{xy}
\usepackage{enumerate}
\usepackage[english]{babel}
\usepackage[bookmarks=true]{hyperref}
\usepackage{tikz}
\usepackage{tikz-cd}
\usepackage{blkarray}
\newtheorem{theo}{Theorem}[section]
\newtheorem{coro}[theo]{Corollary}
\newtheorem{prop}[theo]{Proposition}
\newtheorem{lemm}[theo]{Lemma}

\theoremstyle{definition}
\newtheorem{defi}[theo]{Definition}

\theoremstyle{remark}
\newtheorem{rema}[theo]{Remark}

\newcommand{\bb}[1]{\mathbb{#1}}
\newcommand{\al}[1]{\mathcal{#1}}
\newcommand{\sr}[1]{\mathscr{#1}}

\newcommand{\ra}{\rightarrow}

\newcommand{\x}[1]{\text{#1}}

\title{Hecke transformation for orthogonal bundles over curves}
\author{Christian Pauly}
\address{Laboratoire J.A. Dieudonnée, Université Côte d'Azur, Parc Valrose, 06108 Nice
Cedex 02, France}
\email{pauly@unice.fr}

\author{Hacen Zelaci}
\address{Mathematical departement, El Oued University, N48, 39000, EL Oued, Algeria}
\email{zelaci-hacen@univ-eloued.dz}

\date{}

\begin{document}
\begin{abstract}
    Given an orthogonal bundle $E$ over a smooth projective curve $X$ we define a Hecke transformation in the moduli space of orthogonal bundles by performing an elementary transformation with respect to a Lagrangian submodule $L \subset E_{2x}$ at some point $x \in X$. We show that the analogue of Tyurin's duality theorem holds for 
    orthogonal bundles. Special cases of orthogonal bundles of ranks $2,3,4$ and $6$ are studied in detail. 
\end{abstract}
\maketitle

\section{Introduction}
Hecke transformation is a powerful tool in studying the geometry of moduli spaces of vector bundles over curves, and it is a source of numerous deep results. Since its introduction in the 1970s \cite{NR}, \cite{NR2}, \cite{Ty}, it has been studied intensively, see for example \cite{Hw}, \cite{HwRa}, \cite{S}. Hecke transformations and the induced Hecke correspondences between vector bundles are also one of the fundamental concepts appearing in the formulation of the geometric Langlands correspondence \cite{Frenkel2007}.

In this paper, we introduce a novel definition of Hecke transformation for orthogonal bundles, extending previous  constructions given in \cite{BG}, \cite{Abe} and \cite{CCL}. In order to state the main result we need to introduce some notation.
Let $X$ be a smooth projective curve of genus $g$ defined over an algebraically closed field $\mathbb{K}$ of characteristic $\not=2$ and $E$ an orthogonal bundle over $X$ of rank $r$, i.e.,  a vector bundle with a 
non-degenerate symmetric bilinear form  $b :E \times E \longrightarrow  \al O_X$. Then $b$ 
induces an isomorphism 
$E \stackrel{\sim}{\rightarrow} E^*$. Fix a point $x \in X$ and consider its local ring 
$\mathcal{O} := \mathcal{O}_{X,x}$ with maximal ideal $\mathfrak{M}$ and the stalk $\widehat{E}_x$ of the vector bundle
$E$ at $x$. Then $E_{2x} = \widehat{E}_x \otimes \mathcal{O}/\mathfrak{M}^2$ is a free rank $r$ module
over the ring of dual numbers $\mathbb{K}[\epsilon] = \mathcal{O}/\mathfrak{M}^2$ and the restriction of $b$ endows this module with a 
non-degenerate quadratic form. Our first result is

\begin{theo}\label{maintheorem}
Let $L\subset E_{2x}$ be a Lagrangian submodule and denote by $E'$ the kernel subsheaf of $E$
$$E' = \mathrm{ker} \left(E \twoheadrightarrow E_{2x}/L \right).$$ 
Then $\mathcal{H}(E, L) := E' \otimes \mathcal{O}(x)$ is an orthogonal bundle. 
\end{theo}

The subvariety ${\rm O}\Sigma(E_{2x})$ of the orthogonal Grassmannian
$\mathrm{OGr}(r,E_{2x})$ parameterizing Lagrangian $\mathbb{K}[\epsilon]$-submodules of $E_{2x}$ admits (Proposition \ref{decOsigma}) a stratification $${\rm O}\Sigma(E_{2x}) = \bigsqcup_{i= 0}^k {\rm O}\Sigma_i(E_{2x}),$$ with $k = \lfloor \frac{r}{2} \rfloor$, where each stratum ${\rm O}\Sigma_i(E_{2x})$ is the total space of a vector bundle over the orthogonal Grassmannian $\mathrm{OGr}(i,E_{x})$. We also show in subsection 4.3 that the orthogonal Grassmannian $\mathrm{OGr}(2,E_{x})$
parameterizes the orthogonal analogue of Hecke curves.

We then construct a natural analogue of the Hecke correspondence for orthogonal bundles. To do so, we introduce an open subset 
$M^0$ of the moduli space of semi-stable special orthogonal bundles $\mathcal{M}(\mathrm{SO}_r)$ and construct a parameter space 
${\rm O}\Sigma^0$ of Hecke data with two morphisms
$$
\xymatrix{
& {\rm O}\Sigma^0 \ar[ld]_\pi \ar[rd]^{\pi'} & \\
M^0 & & M^0
}
$$
Then ${\rm O}\Sigma^0$ is a double fibration over $M^0$ with the property that for any orthogonal bundle $E \in M^0$ the
fiber $\pi^{-1}(E)$ is an open subset in ${\rm O}\Sigma_k(E_{2x}) \cup {\rm O}\Sigma_{k-1}(E_{2x})$. The morphism
$\pi'$ associates to a pair $(E,L)$ the Hecke transform $\mathcal{H}(E,L)$ defined in Theorem 1.1. Note that $M^0$, as well
as the fibers $\pi^{-1}(E)$ for any $E \in M^0$, have two connected components. In Proposition \ref{StiefelwhitneyHEL} we express
the connected component containing $\mathcal{H}(E,L)$ in terms of those containing $E$ and $L$.

Finally we show that Tyurin's duality theorem  (\cite[\S 5.2]{Ty}) for ordinary vector bundles can be adapted to orthogonal bundles.

\begin{theo}
The relative tangent bundles over $\rm{O}\Sigma^0$ of the two fibrations $\pi$ and $\pi'$ are dual to each other, i.e. we have
an isomorphism 
$$T_{\pi} = T^*_{\pi'}.$$
\end{theo}

The paper is organized as follows: in Section 2 we start with studying $\mathbb{K}[\epsilon]$-submodules 
of the free $\mathbb{K}[\epsilon]$-module $V \oplus \epsilon V$ for a $\mathbb{K}$-vector space $V$ and show 
that the parameter space $\Sigma \subset \mathrm{Gr}(n,V \oplus \epsilon V)$ of such $\mathbb{K}[\epsilon]$-submodules admits a natural stratification (Proposition \ref{decsigma}). In Section 3 we introduce Lagrangian $\mathbb{K}[\epsilon]$-submodules of the quadratic module $E_{2x}$
and give the orthogonal analogue (Proposition \ref{decOsigma}) of the stratification of the parameter space ${\rm O}\Sigma(E_{2x}) \subset \mathrm{OGr}(r,E_{2x})$. In Section 4 we prove Theorem 1.1 and work out the orthogonal analogue of a Hecke curve. In Section 5 we give
an explicit description of the Hecke transformation for low rank vector bundles. Finally, in Section 6, we prove the
analogue of Tyurin's duality theorem.

We expect these constructions to adapt as well to symplectic bundles and to anti-invariant bundles, as introduced and
studied in \cite{Zel} and \cite{Z2}. We plan to address these questions in a future work.

\bigskip

{\bf Acknowledgments.} The second author acknowledges funding from the CNRS (Chaire Maurice Audin) for a research visit 
at the Laboratoire J.A. Dieudonné at the Université Côte d'Azur in November 2022, when this project was started. We would like to thank 
Vladimiro Benedetti and Indranil Biswas for useful discussions. 

\bigskip

\section{Hecke transformation and modules over the ring of dual numbers}

\subsection{Preliminaries on modules over the ring of dual numbers}
Let $x$ be a point of $X$ and let $\mathcal{O} := \mathcal{O}_{X,x}$ be the local ring of the curve $X$ at the point $x$ and $\mathfrak{M} 
\subset \mathcal{O}$ its maximal ideal. Then $\mathcal{O}/\mathfrak{M} = \mathbb{K}$ and $\mathcal{O}/\mathfrak{M}^2 = \mathbb{K}[\epsilon]$ with 
$\epsilon^2=0$, the ring of dual numbers over $\mathbb{K}$. Note that the last equality depends on the choice of a uniformizing parameter in 
$\mathcal{O}$, which is mapped to $\epsilon$. We fix such an isomorphism through the whole paper and we will write $\mathbb{K}[\epsilon]$ instead
of the ring $\mathcal{O}/\mathfrak{M}^2$. Then we have the following natural exact sequence of $\mathbb{K}[\epsilon]$-modules
$$ 0 \longrightarrow \epsilon \mathbb{K} \longrightarrow  \mathbb{K}[\epsilon] \longrightarrow \mathbb{K} \longrightarrow 0. $$ 
Tensoring with the stalk $\hat{E}_x$ of the vector bundle $E$ at the point $x$ we obtain an exact sequence of $\mathbb{K}[\epsilon]$-modules
$$ 0 \longrightarrow \epsilon E_x \longrightarrow  E_{2x} \longrightarrow E_x \longrightarrow 0,$$
where we write $E_x = \hat{E}_x \otimes \mathbb{K}$, the fiber of the vector bundle $E$ at the point $x$, and 
$E_{2x} = \hat{E}_x \otimes \mathbb{K}[\epsilon]$. We observe that there exists a non-canonical $\mathbb{K}[\epsilon]$-linear
isomorphism $E_{2x} = E_x \oplus \epsilon E_x$, i.e. the above exact sequence is (non-canonically) split.

\bigskip
Let $L$ be a finite-dimensional $\mathbb{K}[\epsilon]$-module. Giving the structure of a $\mathbb{K}[\epsilon]$-module is equivalent to giving a $\mathbb{K}$-linear endomorphism denoted $\cdot \epsilon$, multiplication with $\epsilon$, on $L$ which is nilpotent of order $\leq 2$, i.e. $ \cdot \epsilon \circ \cdot \epsilon= 0$. Considering then the Jordan decomposition of $\cdot \epsilon$ allows us to show that any $\mathbb{K}[\epsilon]$-module $L$ is isomorphic as $\mathbb{K}[\epsilon]$-module to 
$$ L \cong  \mathbb{K}[\epsilon]^{\oplus f} \oplus  \mathbb{K}^{\oplus g},$$
for some uniquely determined integers $f,g \in \mathbb{N}$. Note that $\mathbb{K}$ is a $\mathbb{K}[\epsilon]$-module
on which $\epsilon$ acts trivially.
\bigskip

We will say that $L$ is free if $g=0$ and that $L$ is almost-free if $g=1$. We consider the following natural submodules
of $L$ given as image and kernel of the multiplication by $\epsilon$
$$L^{(1)} = \mathrm{im}(\cdot \epsilon) \qquad \text{and} \qquad L^{(2)} = \mathrm{ker}(\cdot \epsilon).$$
Then we have the natural inclusions
$$ L^{(1)} \subset L^{(2)} \subset L.$$
Note that the quotient module $ L^{(2)}/L^{(1)} \cong \mathbb{K}^{\oplus g}$. We define the \emph{torsion degree} of a $\bb K[\epsilon]-$module $L$ to be $\dim_{\bb K}(L^{(2)}/L^{(1)}).$

Given two $\mathbb{K}[\epsilon]$-module
$L$ and $M$ and a $\mathbb{K}[\epsilon]$-linear homomorphism $\phi : L \rightarrow M$, then it is easy to see
that $\phi$ preserves the above filtrations, i.e. $\phi(L^{(i)}) \subset M^{(i)}$. So we can consider the induced 
map 
$$ \phi_0 :  L^{(2)}/L^{(1)} \rightarrow M^{(2)}/M^{(1)}.$$
We then define
\begin{equation} \label{defhom0}
\mathrm{Hom}_{\mathbb{K}[\epsilon]}^0(L,M) = \{ \phi \in  \mathrm{Hom}_{\mathbb{K}[\epsilon]}(L,M) \ | \ \phi_0 = 0 \}. 
\end{equation}
Note that if $L$ is free, then  
$\mathrm{Hom}_{\mathbb{K}[\epsilon]}^0(L,M) = \mathrm{Hom}_{\mathbb{K}[\epsilon]}(L,M)$.

\subsection{Parameterizing $\mathbb{K}[\epsilon]$-submodules}

We consider a $\mathbb{K}$-vector space $V$ of dimension $r$ and a $\mathbb{K}[\epsilon]$-submodule
$L \subset V \oplus \epsilon V$ with $\dim L = n \leq 2r$. Then $L$ fits into an exact sequence
\begin{equation} \label{esL}
\begin{tikzcd} 
 0 \ar[r] & \epsilon G    \ar[r] \ar[d, hook] &  L  \ar[r]  \ar[d, hook] &  F  \ar[r]  \ar[d, hook] &  0   \\
0 \ar[r] & \epsilon V \ar[r] &  V \oplus \epsilon V  \ar[r,"\pi"]   &  V  \ar[r] &  0,
\end{tikzcd}
\end{equation}
where the vertical maps are inclusions, $\epsilon G = L \cap \epsilon V$ and $F = \pi(L)$. 
One easily sees that the structure of $\mathbb{K}[\epsilon]$-module on $L$ is equivalent to the
inclusion $F \subset G$.

\begin{lemm} \label{equalitytorsiondegree}
        The $\mathbb{K}[\epsilon]$-modules $L$ and $\left(V\oplus \epsilon V \right)/L$ have the same torsion degree.
\end{lemm}

\begin{proof}
    We denote $\left(V\oplus \epsilon V \right)/L$ by $M$. With the above notation we easily see that $L^{(1)} = \epsilon F$ and
    $L^{(2)} = \epsilon G$. Moreover, we have a canonical surjection between two exact sequences of $\mathbb{K}[\epsilon]$-modules
    
    \begin{equation} \label{defj}
    \begin{tikzcd} 
0 \ar[r] & \epsilon V \ar[r] \ar[d,twoheadrightarrow]&  V \oplus \epsilon V  \ar[r,"\pi"] \ar[d,twoheadrightarrow]  &  V  \ar[r] \ar[d,twoheadrightarrow]&  0\\
0 \ar[r] & \epsilon \left( V/G \right) \ar[r,"j"] & M  \ar[r,"\pi"]   &  V/F  \ar[r] &  0,
\end{tikzcd}
   \end{equation}
  from which we deduce that $M^{(1)} = \epsilon \left( V/G \right)$ and $M^{(2)} = \pi^{-1}(G/F)$. In particular we see that 
  $$M^{(2)}/M^{(1)} \cong G/F \cong L^{(2)}/L^{(1)}.$$  
\end{proof}

We introduce the closed subvariety
$$ \Sigma = \{ [L] \in \mathrm{Gr}(n,  V \oplus \epsilon V) \ | \ \epsilon L \subset L \} $$
of the Grassmannian of $n$-dimensional subspaces of  $V \oplus \epsilon V$. Then a point in $\Sigma$ corresponds to 
a $\mathbb{K}[\epsilon]$-submodule L of $V \oplus \epsilon V$ of dimension $n$.

\bigskip

The following proposition describes a natural decomposition of $\Sigma$.

\begin{prop} \label{decsigma}
 We denote $a = \mathrm{max}(0, n-r)$. 
    \begin{enumerate}
        \item We have a decomposition into a disjoint union $$ \Sigma = \bigsqcup_{i=a}^{\lfloor \frac{n}{2} \rfloor} \Sigma_i,$$
        where $\Sigma_i$ is the total space of a vector bundle with fiber $\mathrm{Hom}(F, V/G)$ over the flag
        variety $\mathrm{Flag}(i,n-i, V)$ parameterizing flags $0 \subset F \subset G \subset V$
        with $\dim F = i$ and $\dim G = n-i$.

        \item For any integer $a \leq l \leq \lfloor \frac{n}{2} \rfloor$ the union
        $$ \overline{\Sigma}_l = \bigsqcup_{i=a}^{l} \Sigma_i$$ 
        coincides with the closure of $\Sigma_l$ in $\Sigma$ and we have
        $$ \dim \overline{\Sigma}_l =  n(r-n+l) + l(n-2l).$$
    \end{enumerate}
\end{prop}

\begin{proof}
(1) Consider a $\mathbb{K}[\epsilon]$-submodule $L$ and its associated exact sequence \eqref{esL}.
Denote $i = \dim F$. Then the inclusion $F \subset G$ implies the inequality $i \leq n-i$,
which implies $i \leq \lfloor \frac{n}{2} \rfloor$. Also, the inclusion $G \subset V$ implies
$n - i \leq r$, which is equivalent to $n-r \leq i$. Thus, according to the dimension $i = 
\dim F$ associated with $L$, we get a set-theoretical decomposition of $\Sigma$ into subsets
$\Sigma_i$ for $a \leq i \leq  \lfloor \frac{n}{2} \rfloor$. Fix an integer $i$ in this
interval. Then, the vector spaces $F$ and $G$ of the exact sequence \eqref{esL} give a flag
$0 \subset F \subset G \subset V$, hence a point in $\mathrm{Flag}(i,n-i,V)$. Fixing a flag, we see that all extensions of $F$ by $\epsilon G$ contained in $V \oplus \epsilon V$ are 
parameterized by $\mathrm{Hom}(F, V/G)$ as follows :
\\
For $\varphi \in \mathrm{Hom}(F, V/G)$ denote by $L$ the fiber product given by $\varphi$ and the
natural quotient map $V \rightarrow V/G$. So we obtain the following morphism of exact sequences
$$
\begin{tikzcd} 
 0 \ar[r] & \epsilon G    \ar[r] \ar[d, "\mathrm{Id}"] &  L  \ar[r,"\pi"]  
 \ar[d, "\tilde{\varphi}"] &  F  \ar[r]  \ar[d, "\varphi"] &  0   \\
0 \ar[r] & \epsilon G \ar[r] & \epsilon V  \ar[r]   &  \epsilon \left( V/G \right)  \ar[r] &  0.
\end{tikzcd}
$$
Then the composite map $(\pi, \tilde{\varphi}) : L \rightarrow  F \oplus \epsilon V \subset V \oplus \epsilon V$ 
is injective, with associated flag $0 \subset F \subset G \subset V$.
\\
Conversely, given $L$ we consider the projection $\tilde{\varphi} : L \rightarrow \epsilon V$ onto the factor $\epsilon V$ and observe that $\tilde{\varphi}$ induces the identity on the subspace $\epsilon G$. Thus taking the map induced on the quotients we obtain a map $\varphi : F \rightarrow
\epsilon \left( V/G \right)$.
\\
(2) Fixing an integer $l$ satisfying $a \leq l \leq \lfloor \frac{n}{2} \rfloor$ we observe that the union
$$ \bigsqcup_{i=a}^{l} \Sigma_i = \{  [L] \in \Sigma \ | \  \dim \left( L \cap \epsilon V \right) \geq n - l \}$$
is the intersection in the Grassmannian $\mathrm{Gr}(n, V \oplus \epsilon V)$ of $\Sigma$ with a Schubert variety.
Thus this intersection is a closed subvariety of $\Sigma$. The dimension count is straightforward. 
\end{proof}

\begin{rema}
    Note that the smallest stratum $\overline{\Sigma}_a$ is isomorphic either to
    $\mathrm{Gr}(n, \epsilon V)$ if $a=0$ ($\Longleftrightarrow n \leq r$) or
    $\mathrm{Gr}(a, V)$ if $a=n-r >0$.
\end{rema}

\begin{rema}
In particular we obtain
$$ \dim \Sigma = \begin{cases}
   2k(r-k) & \text{for} \ n= 2k \ \text{even}, \\
   2k(r-k-1) + r -1 & \text{for} \ n = 2k+1 \ \text{odd}. 
\end{cases}$$
\end{rema}

\begin{prop} \label{stratatorsiondegree}
    If $L \in \Sigma_i$, then the torsion degree of $L$ is $n-2i$. In particular, if $L$ lies in the largest
    stratum $\Sigma_{\lfloor \frac{n}{2} \rfloor}$, then $L$ is free if $n$ is even, and $L$ is almost free if $n$ is odd.
\end{prop}

\begin{proof}
    This is a straightforward consequence of the description of $\Sigma_i$ given in Proposition \ref{decsigma}.
\end{proof}

We note that the subvarieties $\overline{\Sigma}_l$ are singular for $a < l \leq \lfloor \frac{n}{2} \rfloor$, but one can construct natural 
desingularisations as follows --- this construction was explained to us by V. Benedetti. Over $\mathrm{Gr}(n-l,V)$, we consider  the relative
Grassmannian
$$ \widetilde{\pi}_l : \widetilde{\Sigma}_l \longrightarrow \mathrm{Gr}(n-l,V) $$
such that over $[G] \in \mathrm{Gr}(n-l,V)$ the fiber $\widetilde{\pi}_l^{-1}([G])$ equals $\mathrm{Gr}(l, G \oplus \epsilon \left(V/G\right))$. Then 
$\widetilde{\Sigma}_l$ is a desingularisation of $\overline{\Sigma}_l$, which follows from

\begin{prop}
For any $a \leq l \leq \lfloor \frac{n}{2} \rfloor$, we have
\begin{itemize}
    \item an open immersion $\Sigma_l \hookrightarrow \widetilde{\Sigma}_l$ over $\mathrm{Gr}(n-l,V)$,
    \item a surjective map $\widetilde{\Sigma}_l \twoheadrightarrow \overline{\Sigma}_l$, which restricts to an isomorphism on $\Sigma_l \subset \widetilde{\Sigma}_l$.
\end{itemize}
\end{prop}

\begin{proof}
It suffices to construct two morphisms $\iota : \Sigma_l \hookrightarrow \widetilde{\Sigma}_l$  and 
$p : \widetilde{\Sigma}_l \twoheadrightarrow \overline{\Sigma}_l$ satisfying $p \circ \iota = \mathrm{id}$.
Given a flag $(F,G) \in \mathrm{Flag}(l,n-l,V)$ and a linear map $\phi: F \ra V/G$, we define 
$$\iota(F,G,\phi) = \widetilde{F} = \mathrm{im}\left(j \oplus \epsilon \phi : F \rightarrow G \oplus \epsilon (V/G) \right),$$
where $j: F \hookrightarrow G$ is the inclusion of $F$ into $G$.  On the other hand, given a subspace $\widetilde{F}
\subset G \oplus \epsilon (V/G)$ we consider the projection $q : G \oplus \epsilon V \rightarrow G \oplus \epsilon (V/G)$
and define 
$$ p(\widetilde{F}) = L = q^{-1}(\widetilde{F}) \subset G \oplus \epsilon V \subset  V \oplus \epsilon V.$$
Then $\epsilon L \subset \epsilon G \subset L$, so $L$ is a $\mathbb{K}[\epsilon]$-module. It is clear that
$p \circ \iota = \mathrm{id}$ and that $p$ is surjective.
\end{proof}

\subsection{Hecke curves revisited}
For the convenience of the reader, we recall here the construction of Hecke curves as given in \cite{NR2} (see also \cite{Hw}). Let $E$ be a rank $r$ vector bundle on $X$ and fix a point $x\in X$. For a hyperplane $G \subset E_x$ we let $$0\ra \widetilde{E}\ra E\ra E_x/G\ra 0.$$ The bundle $\widetilde{E}$ is called the \emph{Hecke transformation} of $E$ at the point $x$ with respect to $G$.  Note that over $x$ we have the exact sequence 
$$0\ra E_x/G \ra \widetilde{E}_x\ra G \ra 0.$$ 
Let $F \subset G$ be a line and consider its $2$-dimensional pre-image $\Pi \subset \widetilde{E}_x$, which 
contains the line $E_x/G$. For any line $L\subset \Pi \subset \widetilde{E}_x$ we consider the Hecke transformation of 
$\widetilde{E}$ with respect to $L$ 
$$0\ra E' \ra \widetilde{E} \ra \widetilde{E}_x/L \ra 0.$$ 
When $L=E_x/G$, the bundle $E'\otimes \mathcal{O}(x)$ coincides with $E$. Moreover, if $g\geqslant 4$ and $E$ is general, the bundle $E'$ is stable (\cite[Proposition 2]{Hw}).  Thus, considering all lines $L \subset \Pi$, we 
obtain a rational curve of vector bundles 
$E' \otimes \mathcal{O}(x)$, parameterized by $\mathbb{P}^1 = \bb{P}(\Pi)$ in the moduli space of stable vector bundles, known as a \emph{Hecke curve}. Note that the parameter space of Hecke curves with fixed $x$ is the Flag
variety $\mathrm{Flag}(1,r-1, E_x)$.

\bigskip

In terms of the decomposition of  
$$\Sigma \subset \mathrm{Gr}(r, E_{2x}) = \mathrm{Gr}(r, E_x \oplus \epsilon E_x)$$
given in Proposition \ref{decsigma} we observe that the Hecke
curve associated to the flag $[F \subset G] \in  \mathrm{Flag}(1,r-1, E_x)$ corresponds
to the fiber of the projection
$$ \Sigma_1 \ra \mathrm{Flag}(1,r-1, E_x)$$
over this point. Note that $\Sigma_1$ is the total space of a line bundle over $\mathrm{Flag}(1,r-1, E_x)$.
Then the fiber $\mathrm{Hom}(F, E_x/G)$ over $[F \subset G]$ can be completed to a $\mathbb{P}^1$ by taking is closure in $\widetilde{\Sigma}_1$.
The point at $\infty$ corresponds to $\Sigma_0 = \{ \epsilon E_x \}$, which correponds under the Hecke transformation to the vector bundle $E$ on the Hecke curve. 

\subsection{Tangent space to the stratum $\Sigma_i$}

In the next proposition we describe the tangent space to the stratum $\Sigma_i$ of 
$\Sigma$.

\begin{prop} \label{desctangentspace}
   Let $[L] \in \Sigma_i$. Then the 
   tangent space at $[L]$ to $\Sigma_i$ is the $\mathbb{K}$-vector space 
    $$T_{[L]} \Sigma_i = \mathrm{Hom}^0_{\mathbb{K}[\epsilon]}(L, (V \oplus \epsilon V)/L)$$
    of $\mathbb{K}[\epsilon]$-linear maps between $L$ and $(V \oplus \epsilon V)/L$
    satisfying $\phi_0 = 0$ (see \eqref{defhom0}). In particular, the injective map 
    $\iota: \Sigma_i \longrightarrow \mathrm{Gr}(n,V\oplus \epsilon V)$ is an embedding.
\end{prop}

\begin{proof}
The natural exact sequence on 
tangent spaces at a point $[L]$ of the fibration $\Sigma_i \longrightarrow \mathrm{Flag}(i, n-i,V)$ is given by 
$$ 0 \longrightarrow T^{rel}_{[L]}\Sigma_i \longrightarrow T_{[L]} \Sigma_i \longrightarrow T_{[F \subset G]}
 \mathrm{Flag}(i, n-i,V) \longrightarrow 0.$$
We can identify the tangent space along the fiber $T^{rel}_{[L]}\Sigma_i$ with $\mathrm{Hom}(F,V/G)$. Forgetting the 
$\mathbb{K}[\epsilon]$-module structure gives an
injective map $\iota: \Sigma_i \longrightarrow \mathrm{Gr}(n,V\oplus \epsilon V)$ and its differential at $[L]$ is
$$ d\iota_{[L]} : T_{[L]} \Sigma_i \longrightarrow \mathrm{Hom}_{\mathbb{K}}(L,(V\oplus \epsilon V)/L). $$
Clearly the restriction of $d\iota_{[L]}$ to $T^{rel}_{[L]}\Sigma_i = \mathrm{Hom}(F,V/G)$ is the injective
map 
$$ \alpha : \mathrm{Hom}(F,V/G) \longrightarrow \mathrm{Hom}_{\mathbb{K}}(L,(V\oplus \epsilon V)/L)$$
sending $\psi \in \mathrm{Hom}(F,V/G)$ to 
$$ j \circ \epsilon \psi \circ \pi : L \stackrel{\pi}{\longrightarrow} F \stackrel{\epsilon \psi}{\longrightarrow}
\epsilon (V/G) \stackrel{j}{\longrightarrow} (V\oplus \epsilon V)/L,$$ 
where $j$ is defined in \eqref{defj}. We introduce the subspace
$$ \mathrm{Hom}^{fil}_{\mathbb{K}}(L,(V\oplus \epsilon V)/L)  \subset 
\mathrm{Hom}_{\mathbb{K}}(L,(V\oplus \epsilon V)/L) $$
of $\mathbb{K}$-linear maps $\psi: L \ra (V\oplus \epsilon V)/L$ satisfying 
$\psi(\epsilon G) \subset \epsilon (V/G)$. Then clearly $\alpha$ factorizes through
$\mathrm{Hom}^{fil}_{\mathbb{K}}(L,(V\oplus \epsilon V)/L)$ and we have an exact sequence 
$$ 0 \rightarrow \mathrm{Hom}(F,V/G) \rightarrow \mathrm{Hom}^{fil}_{\mathbb{K}}(L,(V\oplus \epsilon V)/L)
\stackrel{\beta}{\rightarrow} \mathrm{Hom}(F,V/F) \oplus \mathrm{Hom}(G,V/G) \rightarrow 0.$$
Moreover, the flag variety $\mathrm{Flag}(i, n-i,V)$ can be viewed as a closed subvariety
$$\mathrm{Flag}(i, n-i,V) \subset \mathrm{Gr}(i,V) \times \mathrm{Gr}(n-i,V)$$
and it is straightforward to check that its tangent space at $[F \subset G] \in 
\mathrm{Flag}(i, n-i,V)$ identifies with the subspace of
$$T_{[F]}\mathrm{Gr}(i,V) \oplus T_{[G]}\mathrm{Gr}(n-i,V) = \mathrm{Hom}(F,V/F) \oplus \mathrm{Hom}(G,V/G)$$
cut out by the condition
\begin{equation} \label{descriptiontangentflag}
 (t_F,t_G) \in T_{[F \subset G]}\mathrm{Flag}(i, n-i,V) \ \Longleftrightarrow \ 
t_G \circ i_{(F,G)} = \pi_{(F,G)} \circ t_F, 
\end{equation}
where $i_{(F,G)}: F \hookrightarrow G$ and $\pi_{(F,G)} : V/F \twoheadrightarrow V/G$ are the natural inclusion 
and projection associated with the flag $[F\subset G]$. We will need the following 

\begin{lemm}
    For any  $\phi \in \mathrm{Hom}^{fil}_{\mathbb{K}}(L,(V\oplus \epsilon V)/L)$ we have an equivalence
    $$ \phi \ \text{is} \ \mathbb{K}[\epsilon]-\text{linear} \ \ \Longleftrightarrow \ \ 
    t_G \circ i_{(F,G)} = \pi_{(F,G)} \circ t_F,$$
    with $\beta(\phi) = (t_F,t_G) \in \mathrm{Hom}(F,V/F) \oplus \mathrm{Hom}(G,V/G)$.
\end{lemm}

\begin{proof}
We denote $(V \oplus \epsilon V)/L$ by $M$ and recall from the proof of Lemma \ref{equalitytorsiondegree} that
$L^{(1)} = \epsilon F$ and $M^{(2)} = \pi^{-1}(G/F)$. Since $\phi \in \mathrm{Hom}^{fil}_{\mathbb{K}}(L,(V\oplus \epsilon V)/L)$ we have a morphism of exact sequences
 $$
  \begin{tikzcd} 
0 \ar[r] & \epsilon G \ar[r] \ar[d, "t_G"]&  L  \ar[r,"\pi"] \ar[d,"\phi"]  &  F  \ar[r] \ar[d,"t_F"]&  0\\
0 \ar[r] & \epsilon \left( V/G \right) \ar[r,"j"] & M  \ar[r,"\pi"]   &  V/F  \ar[r] &  0.
\end{tikzcd}
$$
If $\phi$ is $\mathbb{K}[\epsilon]$-linear, then $\epsilon \phi(l) = \phi(\epsilon l)$ for all $l \in L$. We notice
that $\epsilon l \in L^{(1)} = \epsilon F$. If we denote by $\pi(l) = f$, we can write $\epsilon l = \epsilon f$.
On the other hand $\epsilon \phi(l)$ only depends on the class of $\phi(l)$ modulo $M^{(2)} = \pi^{-1}(G/F)$,
which implies that we can consider $\epsilon \phi(l) \in M/M^{(2)} = V/G$. Thus, the equality $\epsilon \phi(l) = \phi(\epsilon l)$ is equivalent to $t_G \circ i_{(F,G)}(f) = \pi_{(F,G)} \circ t_F(f)$ with $\pi(l) =f$.
\end{proof}

We can now conclude the proposition by observing that the previous lemma gives an isomorphism
induced by the differential $d\iota_{[L]}$
$$  T_{[F \subset G]}\mathrm{Flag}(i, n-i,V) \stackrel{\sim}{\longrightarrow} 
\mathrm{Hom}^{fil}_{\mathbb{K}[\epsilon]}(L,(V\oplus \epsilon V)/L)/ \mathrm{Hom}(F, V/G).$$
Finally, we also observe that if $\phi$ is $\mathbb{K}[\epsilon]$-linear, then
\begin{eqnarray*}
 \phi_0 = 0 & \Longleftrightarrow & \phi(L^{(2)}) \subset M^{(1)} \\
            & \Longleftrightarrow & \phi (\epsilon G) \subset \epsilon (V/G) \\
            & \Longleftrightarrow &  \phi \in \mathrm{Hom}^{fil}_{\mathbb{K}[\epsilon]}(L,(V\oplus \epsilon V)/L).
\end{eqnarray*}
Thus $\mathrm{Hom}^{0}_{\mathbb{K}[\epsilon]}(L,(V\oplus \epsilon V)/L) = 
\mathrm{Hom}^{fil}_{\mathbb{K}[\epsilon]}(L,(V\oplus \epsilon V)/L)$ and we are done.
\end{proof}

\begin{coro}
    In particular for $i = \lfloor \frac{n}{2} \rfloor$, we have 
    $$T_{[L]}\Sigma = T_{[L]}\Sigma_{\lfloor \frac{n}{2} \rfloor} =  \mathrm{Hom}^0_{\mathbb{K}[\epsilon]}(L, (V \oplus \epsilon V)/L).$$
\end{coro}

\section{Lagrangian $\mathbb{K}[\epsilon]$-submodules}

\subsection{Preliminaries on Lagrangian submodules}

We now equip $V$ with a non-degenerate symmetric bilinear form 
$$ b_1 :V \times V \longrightarrow  \mathbb{K} $$ 
with associated quadratic form $q_1$. Extending $\mathbb{K}[\epsilon]$-linearly, we obtain a
non-degenerate symmetric bilinear form on $V \oplus \epsilon V = V \otimes_\mathbb{K} \mathbb{K}[\epsilon]$
with associated quadratic form $q_2$.
Then we have the equality 
\begin{equation} \label{quadformE2x}
 q_2(a + \epsilon b) = q_1(a) + \epsilon 2b_1(a,b) \qquad \forall a,b \in V. 
\end{equation}
An important definition of this paper is the following

\begin{defi}
We say that a $\mathbb{K}[\epsilon]$-submodule $L$ of $V \oplus \epsilon V$ is Lagrangian if
\begin{enumerate}
    \item $L$ is isotropic for the quadratic form $q_2$ on $V \oplus \epsilon V$,
    \item $\dim L = \dim V = r.$
\end{enumerate}
\end{defi}

Then we immediately see that for a Lagrangian submodule $L \subset V \oplus \epsilon V$ we have a $\mathbb{K}[\epsilon]$-linear 
isomorphism
$$ (V \oplus \epsilon V) / L \stackrel{\sim} \longrightarrow  L^* = \mathrm{Hom}_{\mathbb{K}[\epsilon]}(L, \mathbb{K}[\epsilon]).$$

Given a subspace $F \subset V$ we denote its associated orthogonal subspace by
$$ F^\perp = \mathrm{ker}( V \stackrel{\sim} \longrightarrow V^* \twoheadrightarrow F^*). $$
Note that we have a canonical isomorphism
$V / F^{\perp} \stackrel{\sim} \longrightarrow  F^*$.
We recall that $F$ is isotropic if and only if $F \subset F^\perp$ and that
$\dim F + \dim F^\perp = r$. 

\begin{lemm} \label{Lagsubmod}
If $L \subset V \oplus \epsilon V$ is a Lagrangian submodule, then we have the following inclusion of exact sequences
$$
\begin{tikzcd}
 0 \ar[r] & \epsilon F^\perp    \ar[r] \ar[d, hook] &  L  \ar[r]  \ar[d, hook] &  F  \ar[r]  \ar[d, hook] &  0   \\
0 \ar[r] & \epsilon V \ar[r] &  V \oplus \epsilon V  \ar[r,"\pi"]   &  V   \ar[r] &  0
\end{tikzcd}
$$
where $F = \pi(L)$ is an isotropic subspace of $V$.
\end{lemm}

\begin{proof}
Using the expression of the quadratic form 
\eqref{quadformE2x} we immediately obtain that $F = \pi(L)$ is an isotropic subspace of $V$. We define
the subspace $M \subset V$ as $\epsilon M  = L \cap \epsilon V$. Then for any $a + \epsilon b \in L$ and
$m \in M$, with $a,b \in V$, we have $a + \epsilon (b+m) \in L$, hence
$$ 0 = q_2(a + \epsilon(b+m)) = q_1(a) + 2 \epsilon ( b_1(a,b) + b_1(a,m) ).$$
We deduce that $b_1(a,m) = 0$ for any $a \in F$ and $m \in M$, hence $M \subset F^\perp$. We then obtain
the equality $M = F^\perp$, since $\dim M = \dim F^\perp$.
\end{proof}

\subsection{Parameterizing Lagrangian submodules}

We can describe the set of Lagrangian 
$\mathbb{K}[\epsilon]$-submodules $L \subset V \oplus \epsilon V$ with fixed $\pi(L) = F \subset V$ as follows.

\begin{prop}\label{submoduleL}
Given an isotropic subspace $F \subset V$ there is a bijective correspondence between the set of 
Lagrangian $\mathbb{K}[\epsilon]$-submodules 
$L \subset V \oplus \epsilon V$ with $\pi(L) = F$ and the vector space 
$$\mathrm{Hom}^{\mathrm{skew}}(F, V/F^{\perp}) = \mathrm{Hom}^{\mathrm{skew}}(F,F^*) = \Lambda^2 F^*,$$
where $\mathrm{Hom}^{skew}(F,F^*)$ denotes the space of skew-symmetric maps $\varphi : F \rightarrow F^*$.
\end{prop}

\begin{proof}
Let $\varphi\in \mathrm{Hom}^{\mathrm{skew}}(F, V/F^{\perp})$, pulling back the exact sequence $$0\ra F^{\perp}\ra V\ra V/F^{\perp}\ra 0,$$ by $\varphi:F\ra V/F^{\perp}$, we get the diagram $$\xymatrix{0\ar[r] & F^{\perp} \ar@{=}[d]\ar[r]^j& L_\varphi \ar[r]^\pi\ar[d]^i & F\ar[r]\ar[d]^\varphi & 0  \\ 0\ar[r] & F^{\perp} \ar[r]^{j'} & V\ar[r]^{\pi'}  & V/F^{\perp} \ar[r] & 0 },$$
Define $\iota:L_\varphi \ra V \oplus \epsilon V$ by $\iota=(\pi,\epsilon i)$. Then $\iota$ is injective,
since if $\iota(v)=0,$ then $\pi(v)=0$, so $v=j(w)$, but then $0=i(v)=i(j(w))=j'(w)$, so $w=0$ and hence $v=0$.\\
We claim that  $\iota(L_\varphi)$ is an isotropic submodule. Indeed, 
by equation $(\ref{quadformE2x})$ we have for any $\omega \in L_\varphi$
$$q_2(\iota(\omega)) = q_1(\pi(\omega))+ \epsilon 2b_1(\pi(\omega),i(\omega)),$$
but since $F$ is isotropic, $q_1(\pi(\omega))=0.$ Now 
\begin{align*}
    b_1(\pi(\omega),i(\omega))&=b_1(\pi(\omega),\pi'(i(\omega))) \\&= b_1(\pi(\omega),\varphi(\pi(\omega)))\\&= b_1(\varphi^t(\pi(\omega)),\pi(\omega))\\&= b_1(-\varphi(\pi(\omega)),\pi(\omega))\\&= -b_1(\pi(\omega),\varphi(\pi(\omega))).
\end{align*}
Hence $b_1(\pi(\omega),i(\omega))=0$, and so $\iota(L_\varphi)$ is isotropic. Since $\dim L_\varphi = r$, it is Lagrangian. Now clearly $\pi(L_\varphi)= F.$\\
Conversely, by Lemma \ref{Lagsubmod} and by the proof of Proposition \ref{decsigma}(1) any Lagrangian $\mathbb{K}[\epsilon]$-submodule $L$ with $\pi(L) = F$ induces a linear map $\varphi: F \ra V/F^\perp = F^*$. Hence any Lagrangian submodule $L$ is of the form $\iota(L_\varphi)$. It remains to show that $\varphi$ is skew-symmetric. We have by 
equation \eqref{quadformE2x} that for any $f \in F$, $b_1(f, \varphi(f))= 0$. Writing this equality for $f+f'$ for any 
$f,f' \in F$ and after developing we obtain that $b_1(f, (\varphi + \varphi^t)(f')) = 0$. We deduce from
these equalities that $\varphi + \varphi^t = 0$, hence $\varphi$ skew-symmetric. 
\end{proof}

\begin{rema}
 If $\dim F = 1$ there is a unique Lagrangian $\mathbb{K}[\epsilon]$-submodule $L$ with associated $F$. The 
 corresponding orthogonal Hecke transformation (see section 4) associated to this Lagrangian submodule was already studied  in
 \cite{Abe}.
\end{rema}

Now we can state the analogue of Proposition \ref{decsigma} for the subvariety 
\begin{equation} \label{definitionOSigma}
 \mathrm{O}\Sigma = \{ [L] \in  \mathrm{OGr}(r, V \oplus \epsilon V) \ | \   \epsilon L \subset L \}
\end{equation}
of the orthogonal Grassmannian $\mathrm{OGr}(r, V \oplus \epsilon V)$ of Lagrangian submodules. First, we note that
$\epsilon V$ is Lagrangian, so $[\epsilon V] \in \mathrm{O}\Sigma$.  Secondly, we recall the fact that 
$\mathrm{OGr}(r, V \oplus \epsilon V)$ has two (isomorphic) connected components, which we denote by
$\mathrm{OGr}^{(m)}(r, V \oplus \epsilon V)$ for $m=0,1$ with the characterizing condition 
$$[\epsilon V] \in \mathrm{OGr}^{(0)}(r, V \oplus \epsilon V).$$
Accordingly, the subvariety $\mathrm{O} \Sigma$ has two connected components
$$\mathrm{O} \Sigma = \mathrm{O}\Sigma^{(0)} \cup \mathrm{O}\Sigma^{(1)}.$$

\begin{prop} \label{decOsigma}
 \begin{enumerate}
    \item  For $m=0,1$, we have a decomposition into a disjoint union 
    $$ \mathrm{O}\Sigma^{(m)} = \bigsqcup_{i \equiv m   [2] \atop{i \leq  \lfloor \frac{r}{2} \rfloor} }  \mathrm{O}\Sigma_i,$$
    where $\mathrm{O}\Sigma_i$ is the total space of a vector bundle with fiber $\Lambda^2 F^*$ over the
    orthogonal Grassmannian $\mathrm{OGr}(i, V)$ parameterizing isotropic subspaces $F$ with $\dim F = i$.
    \item For any integer $0 \leq l \leq \lfloor \frac{r}{2} \rfloor$ the union
    $$ \overline{\mathrm{O}\Sigma}_l = \bigsqcup_{i \equiv l  [2]  \atop{i \leq l} } \mathrm{O}\Sigma_i$$
    coincides with the closure of $\mathrm{O}\Sigma_l$ in  $\mathrm{O}\Sigma^{(m)}$, where $m$ is determined
    by the condition $m \equiv l [2]$.
 \end{enumerate}
\end{prop}

\begin{proof}
The proof is similar to the proof of Proposition \ref{decsigma} using the description of the Lagrangian 
submodules given in Proposition \ref{submoduleL}.
\end{proof}

\begin{rema}
If $r > 2$ is even,  $\mathrm{O}\Sigma_{\frac{r}{2}}$ has two connected components. Its closure  
$\overline{\mathrm{O}\Sigma}_{\frac{r}{2}}$ is connected and has two irreducible components.
\end{rema} 

\begin{rema}
We have the following dimensions :
\begin{itemize}
    \item If $r = 2k$ even, then 
$$ \dim \mathrm{O} \Sigma^{(0)} =  \dim \mathrm{O} \Sigma^{(1)} = k(k-1). $$ 
    \item If $r = 2k+1$ odd,  then
    \begin{itemize}
        \item if $k$ even, $\dim \mathrm{O} \Sigma^{(0)} = k^2$ and $\dim \mathrm{O} \Sigma^{(1)} = k^2-1$,
         \item if $k$ odd, $\dim \mathrm{O} \Sigma^{(0)} = k^2-1$ and $\dim \mathrm{O} \Sigma^{(1)} = k^2$.
         \end{itemize}
\end{itemize}
\end{rema}

\bigskip

We note that the subvarieties $\overline{\mathrm{O}\Sigma}_l$ are singular for 
$2 \leq l \leq  \lfloor \frac{r}{2} \rfloor$, 
but one can construct natural 
desingularisations as follows: consider over $\mathrm{OGr}(l,V)$ the relative
Grassmannian
$$ \tilde{\pi}_l : \widetilde{\mathrm{O}\Sigma}_l \longrightarrow \mathrm{OGr}(l,V) $$
such that over $[F] \in \mathrm{OGr}(l,V)$ the fiber $\tilde{\pi}_l^{-1}([F])$ equals 
$\mathrm{OGr}^0(l, F \oplus \epsilon F^*)$ --- here we equip $F \oplus \epsilon F^*$
with the standard hyperbolic form and $\mathrm{OGr}^0(l, F \oplus \epsilon F^*)$ denotes the connected component of the
orthogonal Grassmannian containing the subspace $[F] \in \mathrm{OGr}(l, F \oplus \epsilon F^*)$.
Then $\widetilde{\mathrm{O}\Sigma}_l$ is a desingularisation of $\overline{\mathrm{O}\Sigma}_l$, which follows from

\begin{prop}
For any $0 \leq l \leq \lfloor \frac{r}{2} \rfloor$ we have
\begin{itemize}
    \item an open immersion $\mathrm{O}\Sigma_l \hookrightarrow \widetilde{\mathrm{O}\Sigma}_l$ 
    over $\mathrm{OGr}(l,V)$,
    \item a surjective map $\widetilde{\mathrm{O}\Sigma}_l \twoheadrightarrow \overline{\mathrm{O}\Sigma}_l$, which restricts to an isomorphism on $\mathrm{O}\Sigma_l \subset \widetilde{\mathrm{O}\Sigma}_l$.
\end{itemize}
\end{prop}

\begin{proof}
It suffices to construct two morphisms $\iota : \mathrm{O} \Sigma_l \hookrightarrow \widetilde{\mathrm{O}\Sigma}_l$  and 
$p : \widetilde{\mathrm{O}\Sigma}_l \twoheadrightarrow \overline{\mathrm{O}\Sigma}_l$ satisfying $p \circ \iota = \mathrm{id}$.
Given an isotropic subspace $F \in \mathrm{OGr}(l,V)$ and a skew-symmetric map $\phi: F \ra F^*$, we define 
$$\iota(F,\phi) = \widetilde{F} = \mathrm{im}\left(\mathrm{id} \oplus \epsilon \phi : F \rightarrow F \oplus \epsilon F^* \right).$$
On the other hand, given a subspace $\widetilde{F} \subset F \oplus \epsilon F^*$ we consider the projection 
$q : F \oplus \epsilon V \rightarrow F \oplus \epsilon \left( V/F^{\perp} \right) =  F \oplus \epsilon F^*$
and define 
$$ p(\widetilde{F}) = L = q^{-1}(\widetilde{F}) \subset F \oplus \epsilon V \subset  V \oplus \epsilon V.$$
Then $L$ is Lagrangian and since $\epsilon L \subset \epsilon F \subset \epsilon F^\perp \subset L$,  $L$ is a $\mathbb{K}[\epsilon]$-module. It is clear that $p \circ \iota = \mathrm{id}$ and that $p$ is surjective.
\end{proof}

\subsection{Tangent space to the largest stratum of $\mathrm{O}\Sigma^{(m)}$}

The next proposition can be proved by adapting without difficulty the proof of Proposition 
\ref{desctangentspace} to the orthogonal case.

\begin{prop} \label{largeststratumissmooth}
   Let $L$ be a Lagrangian $\mathbb{K}[\epsilon]$-submodule of $V \oplus \epsilon V$ 
   such that $[L]$ lies in the largest stratum of
   $\mathrm{O}\Sigma^{(m)}$. Then $[L] \in \mathrm{O}\Sigma^{(m)}$ is a smooth point and the tangent space
   at $[L]$ to $\mathrm{O}\Sigma^{(m)}$ is the $\mathbb{K}$-vector space 
   $$T_{[L]} \mathrm{O}\Sigma^{(m)} = \mathrm{Hom}_{\mathbb{K}[\epsilon]}^{0,skew}(L, L^*)$$
    of skew-symmetric $\mathbb{K}[\epsilon]$-linear maps between the two $\mathbb{K}[\epsilon]$-modules $L$ and $L^*$, which satisfy  $\phi_0 = 0$.
\end{prop}

\begin{rema} \label{remlargeststratum}
  If we denote $k = \lfloor \frac{r}{2} \rfloor$, then the largest stratum of $\mathrm{O}\Sigma^{(m)}$ is either
  $\mathrm{O}\Sigma_k$ or $\mathrm{O}\Sigma_{k-1}$ depending on the parity of $m$ and $k$ --- see Proposition 
  \ref{decOsigma}.
\end{rema}

\section{Orthogonal Hecke transformations and orthogonal Hecke curves}

We first recall that given an orthogonal vector bundle $E$ and a point $x \in X$ we have a free $\mathbb{K}[\epsilon]$-module
$E_{2x}$ and we consider as in \eqref{definitionOSigma} the subvariety
$$ \mathrm{O}\Sigma(E_{2x}) = \{ [L] \in  \mathrm{OGr}(r, E_{2x}) \ | \   \epsilon L \subset L \}. $$
Note that the choice of a $\mathbb{K}[\epsilon]$-linear isomorphism $E_{2x} \cong E_x \oplus \epsilon E_x$ induces an isomorphism
$\mathrm{O}\Sigma(E_{2x}) \cong \mathrm{O}\Sigma$.

\subsection{Proof of Theorem \ref{maintheorem}}

For $L \in  \mathrm{O}\Sigma(E_{2x})$ we have the defining exact sequence for the vector bundle $E'$
$$0\longrightarrow E' \longrightarrow E \longrightarrow E_{2x}/L \longrightarrow 0.$$
On the other hand $E(-2x)$ appears as a subsheaf of $E'$ and we have the exact sequence
$$ 0\longrightarrow E(-2x) \longrightarrow E' \longrightarrow L \longrightarrow 0.$$
Tensoring with $\mathcal{O}(2x)$ and dualizing leads to 
$$ 0\longrightarrow \left( E'(2x) \right)^* \longrightarrow E^* \longrightarrow L^* \longrightarrow 0,$$
where $L^*$ denotes the $\mathbb{K}[\epsilon]$-module 
$\mathrm{Hom}_{\mathbb{K}[\epsilon]}(L, \mathbb{K}[\epsilon])$. Since $L$ is Lagrangian, we have a
$\mathbb{K}[\epsilon]$-linear isomorphism $E_{2x}/L \stackrel{\sim}{\rightarrow} L^*$.
Thus the isomorphism $E \stackrel{\sim}{\rightarrow} E^*$ induces an isomorphism
$E' \stackrel{\sim}{\rightarrow}  \left( E'(2x) \right)^*$, i.e. $E'(x)$ is an orthogonal bundle, as can be seen
from the following exact sequences
$$\xymatrix{0\ar[r] & E' \ar[d]^\sim \ar[r] & E \ar[r]\ar[d]^\sim & E_{2x}/L \ar[r]\ar[d]^\sim & 0  \\ 0\ar[r] & 
 \left( E'(2x) \right)^* \ar[r] & E^* \ar[r]  & L^* \ar[r] & 0. }$$
\hfill\qedsymbol{}
\bigskip
\begin{rema}
We note that the bundle \( E' \) can be constructed via two consecutive Hecke transformations as follows: 

First, let  $L\in {\rm O\Sigma}_i(E_{2x})$ be a Lagrangian subspace and let $F \subset E_x$
defined by $L^{(1)} = \epsilon F$. Then we have $\epsilon F^\perp=L^{(2)}$.  
Consider the first Hecke transformation of \( E \) with respect to \( F \), given by the short exact sequence:  
\[
0 \to \tilde{E} \to E \to E_x / F \to 0.
\]  
The fiber of \( \tilde{E} \) at \( x \) fits into the following diagram $$\begin{tikzcd}
    0\ar[r] & \epsilon (E_x/F) \ar[d,equal] \ar[r] & \widetilde{E}_x \ar[r]\ar[d,hook] & F\ar[d,hook]\ar[r] & 0  \\ 0\ar[r] & 
 \epsilon (E_{x}/F) \ar[r] & E_{2x}/\epsilon F\ar[r]  & E_x \ar[r] & 0,
\end{tikzcd}$$
from which we deduce the following exact sequences $$\begin{tikzcd}
    0\ar[r]& \epsilon (F^\perp/F) \ar[r]\ar[d,hook]& L/\epsilon F\ar[r] \ar[d,hook]& F\ar[r]\ar[d,equal]&0\\ 
        0\ar[r] & \epsilon (E_x/F) \ar[r] & \widetilde{E}_x \ar[r] & F\ar[r] & 0.  
\end{tikzcd}$$ Note that $\dim L/\epsilon F=r-i$. Perform now a second Hecke transformation on \( \widetilde{E} \) with respect to \( G =L/\epsilon F \), resulting in the exact sequence:  
\[
0 \to E' \to \widetilde{E} \to \widetilde{E}_x / G \to 0.
\]
Thus, the bundle \( E' \) is obtained through this process of two consecutive Hecke transformations.
\end{rema}

\begin{rema} \label{reciprocity}
For an orthogonal bundle $E$ and a Lagrangian  submodule $L\in {\rm O\Sigma}_i(E_{2x})$ we denote the associated orthogonal Hecke transformation by $$E'(x)=\al H(E,L).$$
We recall that the quadratic form induces a $\mathbb{K}[\epsilon]$-linear isomorphism $E_{2x}/L 
\stackrel{\sim}{\longrightarrow} L^*$. Thus we have an exact sequence of $\mathbb{K}[\epsilon]$-modules
$$ 0 \longrightarrow L \longrightarrow E_{2x} \longrightarrow L^* \longrightarrow 0.$$
Using the defining exact sequence  of $\al H(E,L)$ we obtain the exact sequence
$$ 0 \longrightarrow L^* \longrightarrow \al H(E,L)_{2x} \longrightarrow L \longrightarrow 0.$$
Moreover, we have the following reciprocity relation 
\begin{equation} \label{reciprocityrelation}
\al H(\al H(E,L),L^*)=E.
\end{equation}
\end{rema}

\subsection{Orthogonal Hecke transformations and second Stiefel-Whitney class}

We recall that the second Stiefel-Whitney class $w_2(E)$ of an orthogonal bundle $E$ takes values in $H^2(X, \mathbb{Z}/2 \mathbb{Z}) =
\mathbb{Z}/2 \mathbb{Z}$.

\begin{prop} \label{StiefelwhitneyHEL}
    For any orthogonal vector bundle \( E \) and any Lagrangian submodule \( L \in O\Sigma_i(E_{2x}) \), the second Stiefel-Whitney classes of $E$ and 
    its orthogonal Hecke transform $\al H(E,L)$ satisfy the following relation  
    \[
    w_2(\al H(E, L)) \equiv w_2(E) + i \ [2].
    \]
\end{prop}
\begin{proof}
    We will use \cite[Theorem 2]{Serre} which says that 
$$w_2(E) \equiv \dim H^0(E \otimes \kappa) + (r+1) \dim H^0(\kappa)  + \dim H^0(\delta \otimes \kappa) \ \ [2]$$
for any theta-characteristic $\kappa$. Here we denote $\delta = \mathrm{det} \ E$. \\
Suppose that $w_2(E) \equiv 0 \ [2]$. Since the second Stiefel-Whitney class is constant under deformation, we can assume that $E$ is of the form  
$$E= \mathcal{O}\otimes V \oplus \delta,$$ 
where $V$ is a vector space of dimension $r-1$ equipped with a non-degenerate quadratic form. Note that $\delta$ is a $2$-torsion 
line bundle, so $\delta$ is an orthogonal line bundle. Choose an isotropic subspace $F \subset V$ of dimension $i$ and a decomposition of the
vector space $V= F \oplus \left( F^{\perp}/F \right) \oplus F'$. Then the orthogonal Hecke transform of $E$ at the point $x$
with respect to a Lagrangian $L \subset V \otimes \mathcal{O}_{2x}$ with $\pi(L) = F$ is 
$$ \al H(E,L) =F \otimes \mathcal{O}(x) \oplus  \left( F^{\perp}/F \right) \otimes \mathcal{O} \oplus F' \otimes \mathcal{O}(-x) \oplus \delta.$$
Now we have \begin{align*}
    w_2(\al H(E,L)) & \equiv \dim H^0(\al H(E,L) \otimes \kappa) + (r+1) \dim H^0(\kappa)  + \dim H^0(\delta \otimes \kappa) \ \  [2] \\
     & \equiv  (r-2i -1)\dim H^0(\kappa)+i\dim H^0(\kappa(x)) + i\dim H^0(\kappa(-x)) \\ 
      & \ \ \ \  +  \dim H^0(\delta \otimes \kappa) + (r+1) \dim H^0(\kappa)  +  \dim H^0(\delta \otimes \kappa)  \ \ [2] \\
      & \equiv  i\left(\dim H^0(\kappa(x))-\dim H^0(\kappa(-x))\right) \ [2] \\
     & \equiv i\left(\dim H^0(\kappa(x))-\dim H^1(\kappa(x))\right) \  [2] \hspace*{1cm}\x{ (by Serre duality)}
    \\& \equiv i \ [2] \hspace*{6cm}\x{ (by Riemann-Roch)}
    \\& \equiv  w_2(E)+i \ [2]. 
\end{align*}
In order to treat the case when $w_2(E) \equiv 1 \ [2]$, we can use $\al H(\al H(E,L),L^*)=E$ ---
see Remark \ref{reciprocity}.
\end{proof}

\subsection{Orthogonal Hecke curves}

Analoguous to the construction of the Hecke curves for vector bundles (see subsection 2.3), we can construct maps from 
$\mathbb{P}^1$ to the moduli of orthogonal bundles. Fix an orthogonal bundle $E$ and a point $x \in X$. With the previous notation 
$$\widetilde{\pi}_2 : \widetilde{\mathrm{O} \Sigma}_2 \longrightarrow \mathrm{OGr}(2, E_x)$$
is a $\mathbb{P}^1$-bundle over the orthogonal Grassmannian $\mathrm{OGr}(2, E_x)$ completing the total space of the line bundle 
$\mathrm{O} \Sigma_2$. For $[F] \in \mathrm{OGr}(2, E_x)$, the fiber $\widetilde{\pi}_2^{-1}([F])$ 
gives, after orthogonal Hecke transformation, a family of orthogonal bundles parameterised by 
$\mathbb{P}^1$. Note that the point at $\infty$ corresponds to 
$\mathrm{O} \Sigma_0 =\{ \epsilon E_x \}$, which corresponds to $E$ under Hecke transformation. Our construction of an
orthogonal Hecke curve associated to $[F] \in \mathrm{OGr}(2, E_x)$ coincides with that of \cite[Section 4.3]{CCL}.

\section{Low rank cases}

In this section we describe the orthogonal Hecke transformations $\al H(E,L)$ for orthogonal vector bundles $E$ having low rank. We fix
a point $x \in X$. We recall that ${\rm O}\Sigma_0 = \{ \epsilon E_x \}$ and that for $L = \epsilon E_x$ we have 
$\al H(E,L) = E$.

\subsection{Rank two bundles}
We observe that ${\rm O}\Sigma_1$ consists of two Lagrangian submodules.
A rank two orthogonal bundle has one of the following two forms:
\begin{enumerate}
    \item $E= M\oplus M^{-1}  \text{ for some }  M \in \mathrm{Pic}(C).$ Note that  $\det E=\mathcal O_C$.
In this case, we clearly see  that  
$$\al H(E,L)= M(x) \oplus M^{-1}(-x) \qquad \text{or} \qquad M(-x) \oplus M^{-1}(x).$$
\item $E=\pi^\alpha_*M \text{ where } \pi^\alpha: C_{\alpha}\xrightarrow{ 2:1 } C $ is the étale double cover attached to the two-torsion
line bundle $\alpha\in\mathrm{Pic}(C)[2]\,$  (i.e.   $\alpha^2=\mathcal O_C$),  and $M\in \rm{Nm}^{-1}(\mathcal O_C)$.
In this case we have that 
$$\al H(E,L)=\pi_*^\alpha(M(x_1-x_2))  \qquad \text{or} \qquad \pi_*^\alpha(M(x_2-x_1)),$$  
where $(\pi^\alpha)^{-1}(x)=\{x_1,x_2\} $.
\end{enumerate}
\bigskip
\subsection{Rank three bundles}
A rank three orthogonal bundle with trivial determinant is of the form $$E={\rm Sym}^2(F)\otimes \det(F)^{-1}={\rm End}_0(F),$$ where $F$ is a rank two bundle. We can assume that $\deg F$ equals $0$ or $1$.\\
Let $L\in {\rm O}\Sigma_1$, then  $L^{(1)}=\bb K\phi$ for some nonzero $\phi\in {\rm End}_0(F)_x={\rm End}_0(F_x)$ such that $\det(\phi)=0$. We then obtain $$L^{(2)}= (\bb K\phi)^\perp=\{\psi\in{\rm End}_0(F_x)\,|\; {\rm Tr}(\psi\circ\phi)=0 \}.$$
Note that since ${\rm Tr}(\phi)=\det(\phi)=0$, we have $\phi^2=0$. In particular, we have ${\rm Im}(\phi)={\rm Ker}(\phi).$\\ 
Let $F'$ be the Hecke transformation of $F$ with respect to  the line $D={\rm Im}(\phi)$. Then we have 
\begin{prop}
    The Hecke transformation of $E$ with respect to $L$ is given by $$\begin{cases}
        \al H(E,L)={\rm End}_0(F'(x)) & \x{ if } \deg F=0 \\  \al H(E,L)={\rm End}_0(F') & \x{ if } \deg F=1
    \end{cases}$$
\end{prop}
\bigskip
\subsection{Rank four bundles}
Let $E$ be an orthogonal bundle of rank $4$ of trivial determinant, then $E=\sr Hom(F,G),$ where $F$ and $G$ are two rank two vector bundles such that $\det F=\det G$. We can assume that $\deg F=0 \x{ or } 1$ which distinguished the two connected components of the moduli space.  Note that the quadratic form is given by the determinant.\\
\begin{enumerate}
    \item {$\underline{ L\in {\rm O}\Sigma_1}:$}
    Let $\bb K\phi\subset E_x$ be an isotropic line. Let $L\subset E_{2x}$ be the Lagrangian subspace $$0\ra \epsilon (\bb K\phi)^\perp\ra L\ra \bb K\phi\ra 0. $$ Denote by $H_x={\rm Ker}(\phi)$ and $H'_x={\rm Im}(\phi)$. Consider the Hecke transformation $$ 0\ra F'\ra F\ra F_x/H_x\ra 0,$$ $$0\ra G'\ra G\ra G_x/H'_x\ra 0. $$ 

\begin{prop}
The Hecke transformation of $E$ with respect to $L$ is given by $$\al H(E,L)=\sr Hom(F',G').$$ 
\end{prop}

    \item {$\underline{ L\in {\rm O}\Sigma_2}:$} Let $L^{(1)}\subset E_x={\rm Hom}(F_x,G_x)$ be the isotropic plane associated to $L$. Then there exists either a line $D\subset G_x$ (or  $D\subset F_x$) such that 
    $$L^{(1)}={\rm Hom}(F_x,D) \;\; \x{(or }L^{(1)}={\rm Hom}(D,G_x)\x{)}.$$ 
    This two cases distinguished the two connected components of the maximal stratum ${\rm O}\Sigma_2$. In the first case we let $M\subset G_{2x}$ such that $L={\rm Hom}(F_{2x},M)$ (or $M\subset F_{2x}$ such that $L={\rm Hom}(M,G_{2x})$\,). Let  $G'$ (resp. $F'$) be the Hecke transformation of $G$ (resp. $F$) with respect to $M$. Then we have 
    \begin{prop}
The Hecke transformation of $E$ with respect to $L$ is given by $$\al H(E,L)=\sr Hom(F,G')\;\; \x{( resp. } \al H(E,L)=\sr Hom(F',G)\;\x{)}.$$ 
\end{prop} 
\end{enumerate}

\bigskip
\subsection{Rank six bundles}
Let $E$ be a rank $6$ orthogonal bundle. There exists a rank $4$ vector bundle $F$ such that 
$$\begin{cases}
E=\Lambda^2F & \x{ with } \det F=\al O_X \x{  or  }\\ E=\Lambda^2(F)(x) & \x{ with } \det F=\al O_X(-2x)   
\end{cases}$$
Note that $\det(\Lambda^2(F))=(\det F)^{\otimes 3}$ and that we have the following surjection $$\pi:\Lambda^2 F_{2x}\twoheadrightarrow (\Lambda^2 F)_{2x}=E_{2x}.$$

Let $L\subset E_{2x}$ be a Lagrangian submodule. Following the dimension of $L^{(1)}=L/\epsilon L$, which is $1$, $2$ or $3$, we get the following three cases:
\begin{enumerate}
    \item {$\underline{ L\in {\rm O}\Sigma_1}:$} In this case, there exists a non-zero $\phi\in E_x={\rm Hom}^{Skew}(F_x,F_x^*)$, isotropic and  such that $$0\ra \epsilon (\bb K\phi)^\perp\ra L\ra \bb K\phi\ra 0. $$

    Moreover, there exists a submodule $P\subset F_{2x}$ such that $\pi(\Lambda^2 P)$ is equal to $L$, and we have  $$ 0\ra \epsilon F_x\ra P\ra K\ra 0,$$ where  $K=\ker \phi\subset F_x$ which has dimension $2$.\\

\item{$\underline{ L\in {\rm O}\Sigma_2}:$}
In this case, there exists a non-zero $w\in F_x$    and $W\subset F_x$ an hyperplane containing $w$ such that $$0\ra \epsilon(w\wedge W)^\perp\ra L\ra w\wedge W\ra 0. $$ 
Moreover, one can sees that $L$ equals $\pi(\Lambda^2 P)$, where $$\begin{tikzcd}
0 \ar[r] 
  & \epsilon W \ar[r] \ar[d, hook] 
  & {P} \ar[r] \ar[d, hook] 
  & \bb Kw \ar[r] \ar[d, hook] 
  & 0 \\
0 \ar[r] 
  & \epsilon F_x \ar[r] 
  & F_{2x} \ar[r] 
  & F_x \ar[r] 
  & 0.
\end{tikzcd}$$

\item{$\underline{ L\in {\rm O}\Sigma_3^{(0)}}:$} In this case there exists $W\subset F_x$ hyperplane such that  $L=\pi(\tilde{P})$ wehre $\tilde P= \Lambda^2{P})$ and $$\begin{tikzcd}
0 \ar[r] 
  & \epsilon W \ar[r] \ar[d, hook] 
  & {P} \ar[r] \ar[d, hook] 
  & W\ar[r] \ar[d, hook] 
  & 0 \\
0 \ar[r] 
  & \epsilon F_x \ar[r] 
  & F_{2x} \ar[r] 
  & F_x \ar[r] 
  & 0.
\end{tikzcd}$$

\item{$\underline{ L\in {\rm O}\Sigma_3^{(1)}}$ }There exists  $w\in F_x$ non zero such that $L=\pi(\tilde{P}),$ where $\tilde{P}=P\wedge F_x$ and $$\begin{tikzcd}
0 \ar[r] 
  & \epsilon \bb Kw \ar[r] \ar[d, hook] 
  & {P} \ar[r] \ar[d, hook] 
  & \bb Kw \ar[r] \ar[d, hook] 
  & 0 \\
0 \ar[r] 
  & \epsilon F_x \ar[r] 
  & F_{2x} \ar[r] 
  & F_x \ar[r] 
  & 0.
\end{tikzcd}$$
\end{enumerate}

\begin{prop}    Let $E=\Lambda^2 F$ $(\x{resp.} (\Lambda^2F)(x)\;)$ and $L\in{\rm O}\Sigma$. Let $F'$ be the Hecke transformation of $F$ with respect to ${P}\subset F_{2x}$.
Then we have 
\begin{enumerate}
    \item $\underline{ L\in {\rm O}\Sigma_1}$: $$\al H(E,L)=(\Lambda^2F')(x) \;(\x{resp. }\Lambda^2(F'(x))\;).$$
    \item $\underline{ L\in {\rm O}\Sigma_2}$: $$\al H(E,L)=\Lambda^2(F'(x)) \;( \x{resp. }(\Lambda^2F'(x))(x)\;).$$
    \item $\underline{ L\in {\rm O}\Sigma_3^{(0)}}$: $$\al H(E,L)=(\Lambda^2F')(x) \;(\x{resp. }\Lambda^2(F'(x))\;).$$
    \item $\underline{ L\in {\rm O}\Sigma_3^{(1)}}$: $$\al H(E,L)=(\Lambda^2F'(x))(x) \;(\x{resp. }(\Lambda^2F'(2x))\;).$$
\end{enumerate} 
\end{prop}

\section{Tyurin's duality theorem}

\subsection{Moduli spaces of orthogonal bundles}

For simplicity of exposition, we assume here that $\mathbb{K}$ has characteristic $0$ and we   consider only 
special orthogonal bundles, i.e., principal $\mathrm{SO}_r$-bundles,
or, equivalently, orthogonal bundles $E$ with trivial determinant. For the precise definition of special orthogonal structure see, e.g., \cite[Remark 2.6]{Ramanan}. Moduli spaces for semistable principal $G$-bundles over $X$ were constructed in \cite{Ramanathan} for 
any semisimple algebraic group $G$. In the particular case $G= \mathrm{SO}_r$ the moduli spaces $\mathcal{M}(\mathrm{SO}_r)$ were studied, e.g., in \cite{Ramanan}, \cite{OS}. We recall that $\mathcal{M}(\mathrm{SO}_r)$ has two connected components
$$ \mathcal{M}(\mathrm{SO}_r) = \mathcal{M}^0(\mathrm{SO}_r)  \cup \mathcal{M}^1(\mathrm{SO}_r)$$
distinguished by the second Stiefel-Whitney class (see also subsection 4.2). We denote by 
$M \subset \mathcal{M}(\mathrm{SO}_r)$ the open dense subset corresponding to stable orthogonal bundles and
consider a universal orthogonal bundle $\mathcal{V}$ over the product $M \times X$\footnote{Strictly speaking
the universal bundle $\mathcal{V}$ only exists over $\widetilde{M}\times X$, where $\widetilde{M} \rightarrow M$ is a finite étale cover, 
but it can be shown that the fibration $\mathrm{O}\Sigma(\mathcal{V}_{2x})$ as defined in \eqref{fibrationOSigma} descends to $M$.}.
The bundle $\mathcal{V}$ comes equipped with a non-degenerate quadratic form $b: \mathcal{V} \times \mathcal{V} \rightarrow \mathcal{O}$ which induces an isomorphism $\mathcal{V} \rightarrow \mathcal{V}^* $.  For a fixed point $x \in X$ we denote by $\mathcal{V}_x$ the restriction 
of $\mathcal{V}$ to $M \times \{ x \}$
\bigskip
Consider the sheaf of principal parts $\mathcal{P}^1(\mathcal{V})$ associated to $\mathcal{V}$ \cite[\S 16.7]{Gro} which fits into the 
exact sequence
$$ 0 \longrightarrow \Omega^1_{M \times X} \otimes \mathcal{V} \longrightarrow \mathcal{P}^1(\mathcal{V}) 
\longrightarrow \mathcal{V} \longrightarrow 0 $$
and denote by $\mathcal{V}_2$ its push-out under the projection of $\Omega^1_{M \times X} =
p_M^*\Omega^1_{M} \oplus p_X^* \Omega^1_{X}$ onto the second factor $p_X^* \Omega^1_{X}$. So 
we obtain an exact sequence of bundles over $M \times X$
\begin{equation} \label{V2}
 0 \longrightarrow p_X^* \Omega^1_{X} \otimes \mathcal{V} \longrightarrow \mathcal{V}_2 
\longrightarrow \mathcal{V} \longrightarrow 0. 
\end{equation}
By functoriality of the formation of the sheaf of principal parts we obtain by restricting this
exact sequence to $\{E \} \times X$ the following
$$ 0 \longrightarrow \Omega^1_{X} \otimes E \longrightarrow \mathcal{P}^1(E)
\longrightarrow E \longrightarrow 0. $$
Moreover, by restricting \eqref{V2} to $M \times \{ x \}$ we obtain the extension
$$ 0 \longrightarrow \epsilon \mathcal{V}_x \longrightarrow \mathcal{V}_{2x} 
\longrightarrow \mathcal{V}_x \longrightarrow 0.$$ 
This extension is not split.
Note that we consider these vector bundles as $\mathcal{O}_M \otimes \mathcal{O}_{X,x}/ \mathfrak{M}^2 = \mathcal{O}_M[\epsilon]$-modules
and that the quadratic form on $\mathcal{V}$
induces an $\mathcal{O}_M[\epsilon]$-linear isomorphism with the exact sequence
$$ 0 \longrightarrow \epsilon \mathcal{V}^*_x \longrightarrow \mathcal{V}^*_{2x} 
\longrightarrow \mathcal{V}^*_x \longrightarrow 0.$$
As in the previous section denote by 
\begin{equation} \label{fibrationOSigma}
\pi : {\rm O}\Sigma(\mathcal{V}_{2x}) \longrightarrow M 
\end{equation}
the fibration in Grassmannians of Lagrangian $\mathbb{C}[\epsilon]$-modules, i.e., for a stable orthogonal bundle $E \in M$ the 
fiber $\pi^{-1}(E)$ equals $\mathrm{O}\Sigma(E_{2x})$.  We recall that both $M$ and $\mathrm{O}\Sigma(E_{2x})$
have two connected components.\\

If we denote $k = \lfloor \frac{r}{2} \rfloor$, then, as observed in Remark \ref{remlargeststratum}, the largest stratum
of the two connected components of  $\mathrm{O}\Sigma(\mathcal{V}_{2x})$ are  
$\mathrm{O}\Sigma_k(\mathcal{V}_{2x})$ and $\mathrm{O}\Sigma_{k-1}(\mathcal{V}_{2x})$, which we denote for
simplicity $\mathrm{O}\Sigma_k$ and $\mathrm{O}\Sigma_{k-1}$. By Proposition \ref{largeststratumissmooth},
$\mathrm{O}\Sigma_k$ and $\mathrm{O}\Sigma_{k-1}$ are smooth fibrations over $M$.
Consider the rational map  \( \al \pi' \) 
\begin{equation} \label{doublefibration}
\xymatrix{
& {\rm O}\Sigma_k \cup {\rm O}\Sigma_{k-1} \ar[ld]_\pi \ar@{-->}[rd]^{\pi'} & \\
M & & M
}
\end{equation}
defined by associating to any stable orthogonal bundle \( E \) and any isotropic submodule  \( L \in {\rm O}\Sigma_l(E_{2x}) \), for $l = k$ or 
$k-1$, the orthogonal Hecke transformation  $\mathcal{H}(E, L)$\footnote{It can be shown that if $E$ has a special orthogonal
structure, then $\mathcal{H}(E, L)$ is naturally equipped with a special orthogonal structure.}.  
Note that \( \al \pi' \) is only a rational map since $\mathcal{H}(E, L)$ is not necessarily stable. This observation leads us
to introduce the following notion.

\begin{defi}
Let $E$ be an orthogonal vector bundle over $X$. We say that $E$ is \emph{superstable} if every
orthogonal Hecke transform of $E$ with respect to any Lagrangian submodule at any point of $X$, is stable.
\end{defi}

\begin{prop}
The orthogonal bundle $E$ is superstable if and only if for every isotropic subbundle $F \subset E$
$$ \mu(F) = \frac{\deg(F)}{\mathrm{rk}(F)} < -1.$$
\end{prop}

\begin{proof}
The proof goes exactly as in \cite[Proposition 9]{Ty}.
\end{proof}    

\begin{prop} \label{opensuperstable}
If $g \geq 4$, then there exists an open dense subset $M^0 \subset M$ corresponding to superstable orthogonal bundles.
\end{prop}

\begin{proof}
By \cite[Lemma 4.1]{CCL} applied with $\delta = 1$, we know that there exists an open dense subset $M' \subset M$ such that any 
$E \in M'$ is $1$-stable, which means that $\frac{\deg(F) + 1}{\mathrm{rk}(F)} < 0$ for any isotropic 
subbundle $F \subset E$. In particular, for an isotropic line subbundle $F \subset E$ we obtain that $\mu(F) = \deg(F) < -1$. 

Consider now a bundle $E \in M'$ which is not superstable. Then there exists an isotropic subbundle $F \subset E$  of 
rank $r'$ with 
$\mu(F) = \frac{d'}{r'} \geq -1$. By the previous observation, we can conclude that $r' \geq 2$. Consider the bundle
$G = \Lambda^2 F \otimes K_X$. Then $\mathrm{rk}(G) = \frac{r'(r'-1)}{2}$ and $\deg(G) = (r'-1)(d' + r'(g-1))$. Thus
\begin{eqnarray*}
 \chi(G) & = &  (r'-1)(d' + r'(g-1)) + \frac{r'(r'-1)}{2}(1-g) \\
        & = & (r'-1) (d' + r'(\frac{g-1}{2})) \\
        & \geq & (r'-1)r'(\frac{g-3}{2}) > 0.
\end{eqnarray*}
We deduce that $\dim H^0(X,G) > 0$, so there exists a non-zero skew-symmetric map $h: F^* \rightarrow F \otimes K_X$. Then 
the composition
$$
\begin{tikzcd}
i \circ h \circ i^* : E \cong E^* \ar[r,"i^*"] & F^* \ar[r,"h"] & F \otimes K_X \ar[r, "i"] & E \otimes K_X 
\end{tikzcd}
$$
is a non-zero skew-symmetric Higgs field, which is nilpotent of order $2$, since $\ker(i^*) = F^\perp$ and $F \subset F^\perp$.
This implies that $E$ is not a very stable orthogonal bundle. But, by \cite[Corollary 5.6]{BR} a general orthogonal bundle is very stable.
Thus we conclude that if $E \in M'$ is very stable, then $E$ is also superstable. This yields the desired open dense subset $M^0 \subset M'\subset M$.
\end{proof}

    
\subsection{The duality theorem}
By construction of the open dense subset $M^0$ in Proposition \ref{opensuperstable} we obtain
a morphism induced by $\pi'$
$$ \pi': \pi^{-1}(M^0) \longrightarrow M $$
and define the open subset ${\rm O}\Sigma^0 = (\pi')^{-1}(M^0) \subset {\rm O}\Sigma_k \cup {\rm O}\Sigma_{k-1}$. Thus
by restricting \eqref{doublefibration} we obtain two morphisms
$$
\xymatrix{
& \rm{O} \Sigma^0 \ar[ld]_\pi \ar[rd]^{\pi'} & \\
M^0 & & M^0
}
$$

\begin{lemm}
For any $E' \in M^0$ the fiber $(\pi')^{-1}(E')$ is an open subset of 
${\rm O}\Sigma_k(E'_{2x}) \cup {\rm O}\Sigma_{k-1}(E'_{2x})$.
\end{lemm}

\begin{proof}
By the reciprocity relation \eqref{reciprocityrelation} we have the equivalence
$$ E' = \mathcal{H}(E,L) \quad \Longleftrightarrow \quad E = \mathcal{H}(E',L^*)$$
for any $E \in M^0$ and any $L \in {\rm O}\Sigma_k(E_{2x}) \cup {\rm O}\Sigma_{k-1}(E_{2x})$. On the other
hand by Lemma \ref{equalitytorsiondegree} and Proposition \ref{stratatorsiondegree} we obtain that $L^* \in {\rm O}\Sigma_k(E'_{2x}) \cup 
{\rm O}\Sigma_{k-1}(E'_{2x})$.
\end{proof}

The main result of this section is  the following:

\begin{theo}[Tyurin's duality]
The relative tangent bundles over $\rm{O}\Sigma^0$ of the two fibrations $\pi$ and $\pi'$ are dual to each other, i.e. we have
an isomorphism 
$$T_{\pi} = T^*_{\pi'}.$$
\end{theo}

\begin{proof}
The proof is similar to \cite[\S 5.2]{Ty}. For the convenience of the reader, we provide a full proof in our case. First,
we will construct over ${\rm O}\Sigma^0 \times X$ a family of orthogonal bundles inducing $\pi'$ as classifying morphism. We can
consider the vector bundle $\mathcal{V}_{2x}$ over $M^0 \subset M$ as a torsion sheaf supported on $M^0 \times \{ x \} \subset
M^0 \times X$. We denote by $\pi^* \mathcal{V}_{2x} \twoheadrightarrow \mathcal{Q}$ the universal quotient bundle over 
${\rm O}\Sigma^0 \times \{ x \} \subset {\rm O}\Sigma^0 \times X$. Then we consider the kernel
$$ \mathcal{F} = \ker \left( (\pi \times \mathrm{id}_X)^* \mathcal{V} \twoheadrightarrow   \pi^* \mathcal{V}_{2x} \twoheadrightarrow \mathcal{Q}  \right).$$
So we have an exact sequence over ${\rm O}\Sigma^0 \times X$
\begin{equation} \label{exseqF}
0 \longrightarrow  \mathcal{F}   \longrightarrow (\pi \times \mathrm{id}_X)^* \mathcal{V}  \longrightarrow  \mathcal{Q}  \longrightarrow 0.
\end{equation}
We then restrict \eqref{exseqF} to ${\rm O}\Sigma^0 \times \mathrm{Spec}(\mathcal{O}/\mathfrak{M}^2) = 
{\rm O}\Sigma^0 \times \{ 2x \}$, i.e. we take the tensor product with the torsion sheaf $\mathrm{pr}_X^*(\mathcal{O}/\mathfrak{M}^2)$,
and then take direct image onto ${\rm O}\Sigma^0 \times \{ x \}$. So we obtain the following long exact sequence
over ${\rm O}\Sigma^0 \times \{ x \} = {\rm O}\Sigma^0$
$$ 0 \longrightarrow \mathrm{Tor}_1^{{\rm O}\Sigma^0 \times X}( \mathcal{Q},\mathrm{pr}_X^*(\mathcal{O}/\mathfrak{M}^2) )  
\longrightarrow \mathcal{F}_{|{\rm O}\Sigma^0 \times \{ 2x \}} \longrightarrow  \pi^* \mathcal{V}_{2x} \longrightarrow \mathcal{Q} \longrightarrow 0.$$
Similarly as in \cite{Ty}, one can show that 
$$ \mathrm{Tor}_1^{{\rm O}\Sigma^0 \times X}( \mathcal{Q},\mathrm{pr}_X^*(\mathcal{O}/\mathfrak{M}^2) )  = \mathcal{Q}. $$
Using the non-degenerate quadratic form on $\mathcal{V}_{2x}$ and the fact that $\mathcal{Q}$ parameterizes quotients by Lagrangian submodules, 
we obtain the following two exact sequences of vector bundles over ${\rm O}\Sigma^0$
$$ 0 \longrightarrow \mathcal{Q}^*  \longrightarrow \pi^* \mathcal{V}_{2x}  \longrightarrow \mathcal{Q} \longrightarrow 0$$
$$ 0 \longrightarrow \mathcal{Q}  \longrightarrow  \mathcal{F}_{|{\rm O}\Sigma^0 \times \{ 2x \}}  \longrightarrow \mathcal{Q}^* 
\longrightarrow 0$$
By construction of $\mathcal{F}$ we see that the classifying map of $\mathcal{F} \otimes \mathrm{pr}_X^* \mathcal{O}(x)$ is the morphism
$\pi'$. We also deduce that the morphism $\pi'$ factorizes as follows
$$
\begin{tikzcd}
{\rm O}\Sigma^0 \ar[d, "\pi"] \ar[rd, "\pi'"] \ar[r, "\psi"] & {\rm O}\Sigma^0 \ar[d, "\pi"] \\
M^0 & M^0
\end{tikzcd}
$$
where $\psi$ is the automorphism of ${\rm O}\Sigma^0$ defined by 
$$ \psi(E,L) = (E', L^*)  $$
with $E' = \mathcal{H}(E,L)$ and $L^* \in {\rm O}\Sigma_k(E'_{2x}) \cup 
{\rm O}\Sigma_{k-1}(E'_{2x})$. Then there exists a line bundle $\mathcal{L}$ over
${\rm O}\Sigma^0$ such that 
$$ (\pi')^*(\mathcal{V}_{2x}) = \mathcal{F}_{|{\rm O}\Sigma^0 \times \{ 2x \}} \otimes \mathcal{L} $$
and such that we have an isomorphism of exact sequences
$$\xymatrix{0\ar[r] & \psi^* \mathcal{Q}^* \ar[d]^\sim \ar[r] & \psi^*\pi^* \mathcal{V}_{2x} \ar[r]\ar[d]^\sim & 
\psi^* \mathcal{Q} \ar[r]\ar[d]^\sim & 0  \\ 
0\ar[r] & \mathcal{Q} \otimes \mathcal{L} \ar[r] & \mathcal{F}_{|{\rm O}\Sigma^0 \times \{ 2x \}} \otimes \mathcal{L}  \ar[r]  & 
\mathcal{Q}^* \otimes \mathcal{L} \ar[r] & 0. }$$
Moreover, by Proposition \ref{largeststratumissmooth} the relative tangent bundle $T_\pi$ equals
$$ T_\pi = Hom_{\mathcal{O}[\epsilon]}^{0,skew}(\mathcal{Q}^*, \mathcal{Q}).$$
On the other hand we have 
$$T_{\pi'} = \psi^* T_\pi = Hom_{\mathcal{O}[\epsilon]}^{0,skew}(\mathcal{Q}, \mathcal{Q}^*).$$
These two bundles are clearly dual to each other.
\end{proof}

\end{document}